\documentclass[12pt]{amsart}
\usepackage{pgfplots}
\usepackage{float}
\usepackage[utf8]{inputenc}
\usepackage{graphicx}
\usepackage{bm}
\usepackage{color}
\usepackage{mdwlist}
\usepackage{amssymb}
\usepackage{soul,xcolor}
\usepackage{ulem}
\allowdisplaybreaks

\usepackage{cleveref}

\newcommand{\beq}{\begin{eqnarray*}}
\newcommand{\feq}{\end{eqnarray*}}
\newcommand{\beqn}{\begin{eqnarray}}
\newcommand{\feqn}{\end{eqnarray}}

\newcommand{\RN}[1]{%
  \textup{\uppercase\expandafter{\romannumeral#1}}%
}

\newtheorem{theorem}{Theorem}[section]
\newtheorem{lemma}[theorem]{Lemma}

\newtheorem{proposition}[theorem]{Proposition}
\theoremstyle{definition}

\theoremstyle{remark}
\newtheorem{remark}[theorem]{Remark}
\numberwithin{equation}{section}

\setstcolor{red}
\newtheorem*{theorem*}{Theorem}

\begin{document}
\title[Critical thresholds in hyperbolic systems]{Critical thresholds in a nonlocal Euler 
system with relaxation}

\author{Manas Bhatnagar and Hailiang Liu}
\address{Department of Mathematics, Iowa State University, Ames, Iowa 50010}
\email{manasb@iastate.edu}
\email{hliu@iastate.edu} 
\begin{abstract} 
We propose and study a nonlocal Euler system with relaxation, which tends to a strictly hyperbolic system under the hyperbolic scaling limit. An independent proof of the local existence and uniqueness of this system is presented in any spatial dimension.   We further derive a precise critical threshold for this system in one dimensional setting. Our result reveals that such nonlocal system admits global smooth solutions for a large class of initial data. Thus, the nonlocal velocity regularizes the generic finite-time breakdown in the pressureless Euler system.  
\end{abstract}

\maketitle


\section{Introduction} \label{prob} 
The question of global smoothness vs. finite time breakdown is fundamental for many hyperbolic balance laws, and  
it was studied in terms of critical threshold phenomena for the first time in \cite{ELT01} for Euler-Poisson equations, followed by critical threshold analysis on various hyperbolic balance laws, see, e.g., 
 \cite{BL19, CCZ16, LL08, LL09j,LL09, LT03, LT04, TT14, WTB12}.
   
In this paper we are concerned with the critical threshold phenomena in 
nonlocal Euler equations with relaxation,  
\begin{subequations}
\label{hypMainwic}
\begin{align}
\label{hypMain}
\begin{aligned}
& \rho_t +\nabla \cdot (\rho Q\ast \mathbf{u})=0,\; x\in \mathbb{R}^N, \; t>0,\\
&  \mathbf{u}_t + (\mathbf{u}\cdot \nabla)\mathbf{u}=\rho(Q\ast \mathbf{u}-\mathbf{u}),\\
\end{aligned}
\end{align}
subject to initial density and velocity, 
\begin{align}
\label{hypMain2}
\begin{aligned}
(\rho (0,x), \mathbf{u}(0,x))= (\rho_0 (x),  \mathbf{u_0} (x)),
\end{aligned}
\end{align}
\end{subequations}
where $\rho_0 \geq 0$.
The nonlocal forces appear in two places in the system, one in the flux for density, and the other in the relaxation for velocity. 
In this context, $Q:\mathbb{R}^N\longrightarrow \mathbb{R}^+\cup \{ 0\}$ is the interaction function, which is assumed 
to be symmetric with bounded total variation in the whole domain.

In one dimension, this system may be seen as a refinement of the convolutional conservation law,
$$
u_t + (u^2/2)_x = Q\ast u-u, 
$$  
for which the critical threshold phenomenon was studied by Liu and Tadmor in \cite{LT01}. Due to the nonlocal nature of this convolution model, only upper and lower thresholds were identified in \cite{LT01}. It remains an open problem whether a sharp threshold can be explicitly obtained for this scalar model.   
The main contribution of this paper is to give a sharp critical threshold for system (\ref{hypMain}) for $N=1$. The multi-case is still an open problem.

To motivate the model, we keep continuous flows in mind and begin with a general physical process and assume that the density  transport  is governed by 
\begin{align}\label{rv}
\rho_t +\nabla\cdot(\rho \mathbf{v})=0.
\end{align}
Here $\mathbf{v} \in \mathbb{R}^N$ represents a mean velocity field. If $\mathbf{v}$ is given in terms of the density variable $\rho$, then  (\ref{rv}) becomes closed.  For example, $\mathbf{v}=-\nabla \delta_\rho E[\rho]$, with certain free energy functional $E[\rho]$, leads to a class of gradient flows in density space, including the heat equation ($E[\rho]=\int \rho {\rm log}\rho dx$), the Fokker-Planck equation ($E[\rho]=\int(\rho {\rm log}\rho +V(x)\rho) dx$), and drift-diffusion models, see, e.g,  \cite{PRS90} in the context of semi-conductor modeling.  In this case, the system is considered to be in local equilibrium and $\mathbf{v}$ is referred to as equilibrium velocity. However, very often $\mathbf{v}$ depends on some extra variables in addition to the conserved ones. The extra variable may be used to characterize non-equilibrium 
features of the system under consideration. Choosing a suitable non-equilibrium variable and determining its evolution equation are the fundamental task of irreversible thermodynamics \cite{GM62, Wa19}. In this paper,  we consider $\mathbf{u}$ as the extra velocity variable and  assume it satisfy the equation involving both nonlinear convection and relaxation:
$$
 \mathbf{u}_t + (\mathbf{u}\cdot \nabla)\mathbf{u}=\rho( \mathbf{v}-\mathbf{u}).
$$
To close the system, we need to relate $\mathbf{v}$ to $\mathbf{u}$, here we choose to use a weighted averaging as
$$
\mathbf{v}=Q\ast \mathbf{u}.
$$  
Such nonlocal structure (non-local velocity) has an analogy with some nonlocal models for 
fluid flows (see, e.g., \cite{FV16, RLM96}). In addition to these modeling considerations,  (\ref{hypMainwic}) has a mathematical interest of its own due to its non-local structure and critical threshold behavior with respect to existence of global solutions vs finite time break-down. 

To be more precise, let us assume some natural properties for the interaction function $Q \geq 0 $: 
\begin{align}
\label{qhyp}
\begin{aligned}
& Q\in W^{1,1}(\mathbb{R}^N),\\
& Q = Q(|x-y|) \; \text{for any}\ x,y\in\mathbb{R}^N \text{and}\ \frac{\partial Q}{\partial r}\leq 0,\ \text{for}\ r> 0, \\
& \int_{\mathbb{R}^N}\!\! Q(x)\, dx = 1.
\end{aligned}
\end{align}
We point out that throughout this paper, any space variable, for example $x,y$, is a vector in $\mathbb{R}^N$, i.e., $x=(x_1,x_2,\ldots ,x_N)$.
To understand the effect of the nonlocal terms involved in the system, we make a hyperbolic scaling
$$
(t,x)\longrightarrow \left( \frac{t}{\epsilon}, \frac{x}{\epsilon}\right), \quad \epsilon > 0,
$$
which leads to a rescaled system of form 
\begin{align}
\begin{aligned}
\label{scale}
& \rho_t + \nabla\cdot (\rho Q^\epsilon\ast \mathbf{u}) = 0,\\
& \mathbf{u}_t + (\mathbf{u}\cdot\nabla ) \mathbf{u} = \frac{\rho}{\epsilon}(Q^\epsilon\ast \mathbf{u}-\mathbf{u}),
\end{aligned}
\end{align}
where $Q^\epsilon(x) = Q(x/\epsilon)/\epsilon^N$ is converging to the delta function $\delta(x) = \Pi_{i=1}^N \delta (x_i)$ 
as the scale parameter  $\epsilon$ tends to zero.
A formal asymptotic analysis in Section \ref{sec2.1} shows that \eqref{scale} tends to a local hyperbolic system
of  form 
\begin{align}
\begin{aligned}
\label{localsys}
& \rho_t +\nabla\cdot (\rho \mathbf{u})=0,\\
&  \mathbf{u}_t + \left( (\mathbf{u} - \rho \boldsymbol\gamma )\cdot\nabla \right) \mathbf{u}=0,
\end{aligned}
\end{align}
where $\boldsymbol\gamma= \int_{\mathbb{R}^N} y Q(y) dy$ is a constant.
Let us illustrate the critical threshold for \eqref{localsys} in 1D case. In such case, it is a strictly hyperbolic system with distinct characteristic speeds $\lambda_1 = u$ and $\lambda_2 = u - \gamma\rho$.  By using the method introduced in 
 \cite{Lax64} to deal with pairs of conservation laws, it can be shown that (\ref{localsys}) will lose $C^1$ smoothness due to the appearance of shock discontinuities unless its two Riemann invariants are nondecreasing, that is 
\begin{align}\label{gg}
u_{0x}(x)+\gamma \rho_{0x}(x)\geq 0 \quad \text{and} \quad u_{0x}(x)\geq 0,\ \forall x\in\mathbb{R}
\end{align}
for (\ref{localsys}). Thus, the finite-time breakdown of (\ref{localsys}) is generic in the sense that it holds for all but a “small” set of initial data.
Such finite time shock formation result also holds true for multi-dimensional hyperbolic conservation law systems under some structural conditions, see \cite[Theorem 7.8.2]{Da16}.  
On the other hand, an additional forcing  if presented in the system will often provide a delicate balance against  
the nonlinear convection, therefore allowing  for a `larger' set of initial data which yield global smooth solutions. 
  Examples,  for which critical thresholds have been established, include the Euler-Poisson equations \cite{BL19, CCZ16, ELT01, TW08, WTB12}, the rotational Euler system \cite{LT04}, the hyperbolic relaxation systems \cite{LL08, LL09j, LL09}, and the Euler-Alignment model \cite{CCTT16, KT18, TT14}, among others. 
 
Critical threshold analysis for scalar hyperbolic balance laws or  $2\times 2$ systems of weakly hyperbolic balance laws is relatively easier. The usual technique is to study the ODE dynamics along the particle path, leading to some differential inequalities, from which threshold conditions are explicitly derived. However, it is often subtle to identify precise threshold conditions for systems of strictly hyperbolic balance laws due to dynamic coupling of distinct characteristic fields \cite{LL09j, TW08}. Also, in such systems including (\ref{hypMain}) breakdown occurs due to shock formation without density aggregation. 

We focus our attention on \eqref{hypMain}. The question is whether the nonlocal feature of the system preserves global regularity for a ``large" set of initial data. 
We answer this question for the 1D case by proving that equation \eqref{hypMain} admits global smooth solutions for a set 
of subcritical initial data if and only if 
$$
 u_{0x}(x) + \rho_0 (x) \geq 0 \;\text{for all} \; x\in\mathbb{R}.
$$
For $N=1$, \eqref{hypMain} appears similar to the Euler Alignment model 
\begin{align*}
& \rho_t +(\rho u)_x=0,\; x\in \mathbb{R}, \; t>0,\\
&  u_t + u u_x=Q\ast (\rho u) -u Q \ast \rho,
\end{align*} 
for which a critical threshold condition of form 
$$
 u_{0x}(x) + (Q\ast \rho_0) (x) \geq 0 \;\text{for all} \; x\in\mathbb{R}
$$
was proved in \cite{CCTT16}, where the authors refined the earlier result in \cite{TT14}. We would like to point out that  these two models are quite different in the sense that 
 the Euler-Alignment model is closer to a weakly hyperbolic system for which both $-u_x$ and $\rho$ blow up simultaneously, and our model is closer to a strictly hyperbolic system as $\rho$ remains bounded for all time even when $-u_x$ blows up at a finite time. Such difference requires novel estimates in our analysis for both local existence and quantifying the critical threshold.   


The rest of this paper is organized as follows. In Section \ref{mainresults} we present local existence results  and obtain the  uniform solution bounds for any spatial dimension. We also show that the asymptotic limit of the rescaled system is actually a strictly hyperbolic system. Section \ref{cta} contains the critical threshold analysis for \eqref{hypMain} in one dimension. This analysis is carried out as an a priori estimate on smooth solutions.  Appendix A is devoted to  the a priori estimates in high norm and  a detailed proof of the local well-posedness is finally given in Appendix B.

{\bf Notation: } Throughout the paper, we denote  $\beta = ||Q||_{W^{1,1}(\mathbb{R}^N)}$, and  $||\cdot ||$ the $L^2(\mathbb{R}^N)$ norm unless specified otherwise. 
For any function $f$, and multi-index $\alpha = (\alpha_1,\alpha_2,\ldots, \alpha_N) \in (\mathbb{Z}_+)^N$, $D_\alpha f = \partial^{\alpha_1}_{x_1} \partial^{\alpha_2}_{x_2}\ldots \partial^{\alpha_N}_{x_N} f$ and $|\alpha| = \alpha_1 +\ldots +\alpha_N$. Also, for any positive integer $k$, $||D^k f||^2 = \Sigma_{|\alpha|=k} ||D_\alpha f||^2$. Any Sobolev space, $W$ is to be interpreted as $W(\mathbb{R}^N)$ unless stated otherwise. And boldface letters and symbols are notations for vectors which are all of dimension $N$.

\section{Preliminaries and solution bounds}
\label{mainresults}
\subsection{Asymptotic limit}\label{sec2.1} We begin to characterize (formally) the behavior of system (\ref{scale}) as $\epsilon$ tends to zero.  
Let 
$$
\rho = n + \epsilon\rho_1 + \epsilon^2\rho_2+\ldots,  \quad 
 \mathbf{u} = \mathbf{v} + \epsilon \mathbf{u_1} + \epsilon^2 \mathbf{u_2} + \ldots.
 $$ 
Keeping $Q^\epsilon$ unchanged for the moment, and collecting the leading terms in $\epsilon$,  we obtain
\begin{align*}
& n_{t} + \nabla\cdot (n Q^\epsilon\ast \mathbf{v}) + O(\epsilon) = 0,\\
& \mathbf{v}_{t} + (\mathbf{v}\cdot\nabla ) \mathbf{v} = \frac{n(Q^\epsilon\ast \mathbf{v} - \mathbf{v})}{\epsilon} + n(Q^\epsilon\ast \mathbf{u_1} - \mathbf{u_1}) + \rho_1(Q^\epsilon\ast \mathbf{v} - \mathbf{v}) + O(\epsilon).
\end{align*}
In this way, letting $\epsilon\to 0$, we indeed recover the behavior of the leading order term from $\frac{1}{\epsilon}n(Q^\epsilon\ast \mathbf{v} - \mathbf{v})$.  Noting that $Q^\epsilon$ converges to the delta function, we get
\begin{align*}
& n_{t} + \nabla\cdot (n \mathbf{v})  = 0,\\
& \mathbf{v}_{t} + (\mathbf{v}\cdot\nabla ) \mathbf{v} = n\lim_{\epsilon\to 0} \frac{(Q^\epsilon\ast \mathbf{v} - \mathbf{v})}{\epsilon}.
\end{align*}
To evaluate the limit on the right hand side, 
\begin{align*}
 \frac{1}{\epsilon} \int  \frac{1}{\epsilon^N} Q\left(\frac{x-y}{\epsilon}\right)(\mathbf{v}(y) - \mathbf{v}(x))\, dy
= & \int Q(z)\frac{\mathbf{v}(x+\epsilon z) - \mathbf{v}(x)}{\epsilon}\, dz\\
= &  \int Q(z) (z\cdot\nabla_x ) \mathbf{v} (\zeta)\, dz,
\end{align*}
where $\zeta$ lies on the line joining $x$ and $x+\epsilon z$. 
Assuming $|z|Q(z)\in L^1(\mathbb{R}^N)$, we can use dominated convergence theorem to obtain the limit as $\epsilon\to 0$. Consequently,
$$
\lim_{\epsilon\to 0}\frac{1}{\epsilon} \int  \frac{1}{\epsilon^N} Q\left(\frac{x-y}{\epsilon}\right)(\mathbf{v}(y) - \mathbf{v}(x))\, dy =\left(\left( \int yQ(y)\, dy\right)\cdot\nabla \right)\mathbf{v}.
$$
Plugging this limit back in, we indeed obtain (\ref{localsys}). 

\subsection{Solution bounds}
Since \eqref{hypMain} is a nonlocal system, local existence and uniqueness does not follow from the existing theory of hyperbolic PDE systems. Therefore, we first study the existence of local-in-time classical solutions to the problem \eqref{hypMain}. We prove the following 
\begin{theorem}
\label{local}
$($\textbf{Local existence}$)$ Let $s > \frac{N}{2}$ be a positive integer.  Suppose $\rho_0\in L^\infty (\mathbb{R}^N), \mathbf{u_0} \in (L^\infty (\mathbb{R}^N))^N$ and $D^1\rho_{0}\in H^s (\mathbb{R}^N), D^1\mathbf{u_{0}}\in (H^s(\mathbb{R}^N))^N$. Then for any $M>0$, if 
$$
\max\{ ||\rho_0||_\infty, ||\nabla\rho_0||_{H^s(\mathbb{R}^N)}, ||\mathbf{u_0}||_\infty, ||\nabla\mathbf{u_{0}}||_{H^s(\mathbb{R}^N)}  \}\leq M,
$$ 
then there exists $T>0$, depending on $M,Q$, and continuously differentiable functions $\rho, \mathbf{u}$ satisfying 
$$
\max\left\{ \sup_{[0,T]}||\rho(t,\cdot)||_\infty, \sup_{[0,T]}||\nabla \rho(t,\cdot)||_{H^s(\mathbb{R}^N)}, \sup_{[0,T]}||\mathbf{u}(t,\cdot)||_\infty, \sup_{[0,T]} ||\nabla\mathbf{u}(t,\cdot)||_{H^s(\mathbb{R}^N)}\right\} \leq 2M,
$$ 
which are classical solutions to the problem \eqref{hypMain}.
\end{theorem}
\begin{remark} Instead of considering initial data $\rho_0, \mathbf{u_0} \in H^{s+1}(\mathbb{R}^N)$, as usually done when  using the energy methods to prove local well-posedness. Here we allow more general initial data in the sense that $\rho_0$ can have infinite mass, and both density and velocity may not be decaying at $|x|=\infty$. 
\end{remark}

We prove this result by constructing an approximating sequence to the exact solution. The details of the proof is deferred to Appendix B.


We now move on to proving a priori bounds on $\rho$ and $\mathbf{u}$. We begin by proving a useful lemma.
\begin{lemma}
\label{transbd}
Let $f,h\in C([0,\infty)\times\mathbb{R}^N)$ and $\mathbf{g}\in (C([0,\infty)\times\mathbb{R}^N))^N $ be such that
$$
f_t + \nabla\cdot (f \mathbf{g}) = h.
$$
Let us also assume that $g$ is Lipschitz continous in the second variable. Then
$$
||f(t,\cdot)||_\infty\leq ||f(0,\cdot)||_\infty e^{\int_0^t ||\nabla\cdot \mathbf{g}(s,\cdot)||_\infty ds} + \int_0^t \!\! e^{\int_s^t ||\nabla\cdot \mathbf{g}(\tau,\cdot)||_\infty d\tau}||h(s,\cdot)||_\infty ds,
$$
for all $t>0$.
\end{lemma}
\begin{proof}
On employing method of characteristics on the given equation,
$$
f_t + \mathbf{g}\cdot\nabla f + f\nabla \cdot \mathbf{g} = h,
$$
we have 
$$
\frac{df}{dt} + f(t,x(t,\alpha))\nabla\cdot \mathbf{g} (t,x(t,\alpha)) = h(t,x(t,\alpha)), 
$$
along the curve
$$
\left\{ (t,x):\frac{dx}{dt} = \mathbf{g}(t,x(t,\alpha)),\ x(0;\alpha) = \alpha \right\}.
$$
For any fixed $\alpha$, this is a linear ODE in time and can be explicitly solved to obtain
$$
f(t) = f(0)e^{-\int_0^t \nabla\cdot \mathbf{g}(s)ds}+\int_0^t\!\! e^{-\int_s^t \nabla\cdot \mathbf{g}(\tau)d\tau}h(s)ds.
$$
Taking supremum over $\alpha$ concludes the proof.
\end{proof}

\begin{proposition} 
\label{bdds}
Let $\rho, \mathbf{u}$  be continuously differentiable solutions to \eqref{hypMain} subject to 
initial data $\rho_0\in L^\infty (\mathbb{R}^N),\, \rho_0  \geq 0$ and $\mathbf{u_0}\in (L^\infty (\mathbb{R}^N))^N$.  Then
\begin{itemize}
\item
$($\textbf{Maximum principle on} u$)$ 
\begin{align}
\label{umaxprin}
\inf_{x\in\mathbb{R}^N} \mathbf{u_0}(x) \leq \mathbf{u} \leq \sup_{x\in\mathbb{R}^N} \mathbf{u_0}(x).
\end{align}  
The vector inequality is in the component-wise sense.
\item
$($\textbf{Bounds on }$ \rho)$
$\rho$ remains uniformly bounded for each $t>0$,
\begin{align}
\label{densitybound}
0 \leq \rho(t,x) \leq ||\rho_0||_\infty e^{||Q||_{W^{1,1}(\mathbb{R}^N)} || \mathbf{u_0}||_\infty t},\quad \forall x\in\mathbb{R}^N,\, t>0.
\end{align}
\end{itemize}
\end{proposition}

\begin{proof}
Using Lemma \ref{transbd} on \eqref{hypMain}(i) and $\beta = ||Q||_{W^{1,1}}$, we obtain,
\begin{align}\label{rb}
0 \leq \rho (t,x) \leq ||\rho_0||_\infty e^{\beta\int_0^t || \mathbf{u}(s,\cdot )||_\infty\, ds},\quad \forall x\in\mathbb{R}^N,\, t>0.
\end{align}
Therefore, it suffices to only prove the first assertion of the proposition. We will prove \eqref{umaxprin} for a single fixed $i^{th}$ component of $\mathbf{u}$ using $\partial_t u_i + \mathbf{u}\cdot \nabla u_i = \rho(Q\ast u_i - u_i )$. Let's denote $u_i$ by $u$. Likewise $u(0,x) = i^{th}$ component of $\mathbf{u_0}=: u_0$. Let us suppose for contradiction that $u$ goes outside the interval $[\inf u_0 ,\sup u_0]$ at some positive time. (Without loss of generality, assume it violates the upper bound, similar argument holds for the lower bound as well). 
Consider the first time it happens. So, $\exists t_0\geq 0$ and $x_0$ such that $u(t_0,x_0)=M:=\sup_{x\in\mathbb{R}^N} u_{0}(x)$ and $\forall \delta>0$, $\exists  t_\delta, x_\delta$, with $t_0<t_\delta<t_0+\delta$ and $u(t_\delta ,x_\delta)>M$. \\
We claim that $\nabla u(t_0 ,x_0) = 0$. Because if $\partial_j u(t_0,x_0)$ has a sign for some $j=1,\ldots , N$, then there exists some $x_1 = x_0 \pm \delta e_j$ in the neighborhood of $x_0$ such that $u(t_0, x_1)>M$ which is indeed a contradiction.\\
For the moment, let's assume $\rho(t_0,x_0)>0$. From \eqref{hypMain}(ii),
$$
u_t (t_0,x_0) = \rho(t_0,x_0)\int_{\mathbb{R}^N}Q(x_0-y)\left(u(t_0,y) - u(t_0,x_0)\right)\, dy<0.
$$
Also, for any differentiable curve $X(t)$ with $X(t_0) = x_0$ we have that
$$
\frac{du}{dt}(t_0,X(t_0)) = \left. u_t + \mathbf{u}\cdot\nabla u\right|_{ (t_0, x_0)}  = u_t(t_0,x_0)<0,
$$
and, therefore, $u(t,x)<M$ for any $(t,x)$ sufficiently close to $(t_0,x_0)$ with $t>t_0$.
This results in a contradiction.\\ 
Lastly, if $\rho(t_0,x_0)=0$, let $\overline{u} := u-\epsilon t$ for $\epsilon >0$ fixed. Plugging this in \eqref{hypMain}(ii) for the $i^{th}$ equation, we obtain
$$
\overline{u}_t + \mathbf{u}\cdot\nabla\overline{u} = \rho (Q\ast\overline{u} - \overline{u})-\epsilon,
$$ 
and, therefore, $\overline{u}_t(t_0,x_0)<0$. Along the same line of argument as above, we have $\overline{u}\leq\sup  \overline{u}(0,\cdot) = \sup u_0$. Consequently, $u \leq \sup u_0+\epsilon t$ holds for any $\epsilon>0$. Taking $\epsilon\to 0$ gives one-sided inequality of the maximum principle. Likewise, the other inequality can be obtained by letting $\overline{u} = u + \epsilon t$. This completes the proof of the proposition.
\end{proof}

\begin{remark}
It is well known that the 1D “pressure-less” Euler system, that is (1.1) with $N=1$ and $Q\ast u$ replaced by $u$, admits the aggregation phenomena: the breakdown occurs when $-u_x(t, x)$ and $\rho(t,x)$ approach $+\infty$ 
simultaneously at the critical time, $t\uparrow t_c$. In contrast, the above result tells us that the nonlocal velocity for the density prevents the concentration of the density. Thus, the only breakdown for the full system (1.1) occurs through the formation of shock discontinuities, where $|\nabla \mathbf{u}|$ and/or $|\nabla\rho|$ blow up as $t\uparrow t_c$, 
but $\rho$ will not concentrate at any point.
\end{remark}

\begin{theorem}\label{hsenergyestimate}
$($\textbf{$H^s$ energy estimate}$)$ Let $s>\frac{N}{2}$ be a positive integer. Define $Y(t)= ||\nabla \mathbf{u}(t,\cdot)||^2_{H^s(\mathbb{R}^N)} + ||\nabla\rho(t,\cdot)||^2_{H^s(\mathbb{R}^N)}$. Then  
$$
Y(T)\leq Y(0)\exp\left[ C\int_0^T ||\rho(\tau,\cdot)||_\infty + ||\mathbf{u}(\tau,\cdot)||_\infty + ||\nabla\mathbf{u}(\tau,\cdot)||_\infty\, d\tau \right],
$$
with $C = C(||Q||_{W^{1,1}(\mathbb{R}^N)},s)$. In particular, a unique smooth solution to \eqref{hypMainwic} exists for all time if and only if $\int_0^T  ||\nabla\mathbf{u}(\tau,\cdot)||_\infty\, d\tau < \infty$ for all $T>0$. 
\end{theorem}
We prove this result using a commutator type estimate found in \cite[Proposition 2.1]{M84}.  Details of the proof are 
deferred to Appendix A.  


\section{Critical Thresholds}
\label{cta}
Our main result, investigated in this section, confirms the critical threshold phenomenon by quantifying  
the precise threshold in one dimension. Since this section assumes $N=1$ in \eqref{hypMain}, we replace $\nabla \rho$ with $\rho_x$ and use normal letters instead of boldface letters.  
\begin{theorem}
\label{finalresult}
Consider the nonlocal Euler system with relaxation \eqref{hypMainwic}.
\begin{itemize}
\item{[Subcritical region]} A unique solution $\rho,u\in C([0,\infty);L^\infty(\mathbb{R}))$ and 
$$\rho_x , u_x \in C([0,\infty);H^s(\mathbb{R})), \quad s\geq 1 $$ exists if $u_{0x}(x) + \rho_0 (x) \geq 0$ for all $x\in\mathbb{R}$. 
\item{[Supercritical region]} If $\exists x_0\in\mathbb{R}$ for which $u_{0x}(x_0) < -\rho_0 (x_0)$, then $u_x\to-\infty$ in a finite time.
\end{itemize}
\end{theorem}
\begin{remark} 
This result is sharp in the sense that it provides a precise initial threshold characterized by the point values of $\rho_0$ and $u_{0x}$: either subcritical initial data which evolve into global strong solutions, or supercritical initial data which will lead to a blow up in a finite time. This is reminiscent 
of the critical threshold in 1D Euler-Poisson systems \cite{ELT01}.
\end{remark}

The threshold analysis to be carried out is the a priori estimate on smooth solutions as long as they exist. 
In this section, we will show the presence of a precise critical threshold which divides the initial data into subcritical and supercritical regions. 

We proceed to derive the characteristic system which is essential in our analysis. Differentiate the second equation in (\ref{hypMain})(ii) with respect to $x$ to obtain:
$$
u_{xt} + uu_{xx} + u_x^2  = (\rho Q\ast u )_x - \rho_x u - \rho u_x.
$$
Using \eqref{hypMain}(i) and setting $d = u_x + \rho$, we obtain
	
	\begin{align}\label{nonlin}
	& d' = -d(d-\rho),
	\end{align}
	where 
	$
	\{\}' = \frac{\partial}{\partial t} + u\frac{\partial}{\partial x},
	$
	 denotes the differentiation along the particle path,
	\begin{align}
	\label{chpath}
	\Gamma=\{(t,\RN{10})|\;  d\RN{10}/dt=u(t,\RN{10}(t)),\, \RN{10}(0)=\alpha  \in \mathbb{R}\},
	\end{align}
	and $d_0 := d(0; \alpha) = u_{0x} (\alpha ) + \rho_0(\alpha )$. Note that here $X(t):=X(t; \alpha)$ and $\rho(t):=\rho(t, X(t; \alpha))$.
	
%

\begin{lemma}
\label{boundond}
For the equation \eqref{nonlin} with $\rho(t) \geq 0$, we have the following:
\begin{itemize}
\item
If $d_0\geq 0$, then $d$ remains bounded for all $t>0$. More precisely,
$$
0\leq d(t)\leq \max\{ d_0, \rho(t) \},\quad \forall t>0.
$$
\item
If $d_0<0$, then $\exists t_c$ such that $\lim_{t\to t_c^-} d(t) = -\infty$.
\end{itemize}
\end{lemma}

\begin{proof}
If $d_0=0$, then $d=0$ is the unique solution. If $d_0>0$: whenever $d\in (0,\rho)$, $d'>0$ and hence, $d$ increases. This ensures positivity of $d$. On the other hand, if $d>\rho$ then $d'<0$ and hence, $d$ is decreasing. \eqref{densitybound} ensures $d$ remains finite for any $t>0$. This proves the first assertion.\\ 
Now assume $d_0<0$. Clearly $d'<0$ for all positive times and, therefore, $d\leq d_0<0$ for all $t>0$. Consequently, $d'<-d^2$ which gives
$$
d < \frac{d_0}{1+d_0 t}\quad \forall t>0.
$$  
And $d\to -\infty$ in finite time, $t_c <-1/d_0$, which proves the second assertion.
\end{proof}

We now prove Theorem \ref{finalresult} using this proposition.
\begin{proof}
For any fixed $(t,x)$, there exists a unique $\alpha$ and a curve $\RN{10}(t;\alpha)$ such that  $\RN{10}(t;\alpha)=x$,
 and 
$$
u_x(t, x)+ \rho(t,x) =u_x(t,\RN{10}(t;\alpha))+ \rho(t,\RN{10}(t;\alpha))=d(t),
$$
along each characteristic path \eqref{chpath}.
Suppose $u_{0x}(x) + \rho_0 (x) \geq 0$ for all $x\in \mathbb{R}$,  which corresponds to $d(0) \geq 0$ for all $\alpha$.   Lemma \ref{boundond} gives
\begin{align*}
0 \leq u_x(t, x) + \rho(t, x) & \leq \max\{d_0, \rho(t)\}.
\end{align*}
Note that 
$$
d_0= u_{0x}(\alpha ) + \rho_0 (\alpha) \leq ||u_{0x }||_\infty +\|\rho_0\|_\infty,
$$
and also 
$$
\rho(t):=\rho(t,\RN{10}(t,\alpha)) \leq ||\rho_0||_\infty e^{\beta ||u_0||_\infty t},
$$
where we used \eqref{densitybound}. Thus we obtain 
$$
||u_x(t,\cdot )||_\infty \leq  ||u_{0x}||_\infty + 2||\rho_0||_\infty e^{\beta ||u_0||_\infty t}.
$$
This when combined with the estimate in Theorem \ref{hsenergyestimate} ensures that smooth solution exists for all $t>0$. Now assume $u_{0x}(x_0) + \rho_0 (x_0) < 0$ for some $x_0\in\mathbb{R}$, that is $d_0=d(0; x_0)<0$.  
Lemma \ref{boundond} then gives that 
$$
(u_x+\rho)(t,\RN{10}(t,x_0))\to -\infty
$$ 
in finite time. Subsequently, $u_x\to -\infty$ and solution ceases to be in $H^s(\mathbb{R})$. This concludes the proof.
\end{proof}

\begin{remark} In the presence of $\epsilon$, the corresponding threshold condition for \eqref{scale} is that global smooth solution exists if and only if   
$$
u_{0x}(x) + \frac{1}{\epsilon} \rho_0 (x) \geq 0  \quad \forall  x\in\mathbb{R}. 
$$
It is worth pointing out that a smaller $\epsilon$ enables a larger range for $u_{0x}(x)$, indicating
the smoothing effects of the relaxation.  Interestingly, letting $\epsilon \to 0$ does not seem to yield the threshold condition (\ref{gg}) for the limiting system (\ref{localsys}). This signifies that the limit $\epsilon\to 0$ is singular. This is reminiscent of Euler-Poisson equations with constant background $c\geq 0$, where we do not recover the threshold condition of $c=0$ case on letting $c\to 0$. On the other hand, as $\epsilon \to \infty$, the above threshold reduces to   
$$
u_{0x}(x) \geq 0  \; \forall  x\in\mathbb{R}. 
$$
This is exactly the threshold condition for the inviscid Burgers' equation.
\end{remark}
\begin{remark} 
An attempt has been made to extend the study of critical thresholds to the two-dimensional model, yet the nonlocal coupling in system (\ref{hypMainwic}) brings subtle difficulties in obtaining  an effective control of both the velocity gradient matrix and the density gradient.  The techniques using spectral dynamics of \cite{LT02} for the Euler-Alignment system as presented in \cite{HT17, TT14} do not seem to be adaptable to the present situation. 
\end{remark}

\appendix
\section{A priori estimates in high norm}
We begin by stating some important lemmas which will be used in the proofs of Theorems \ref{hsenergyestimate} and \ref{local}. We will need the following commutator type estimate, \cite[Proposition 2.1]{M84}.
\begin{lemma}
\label{commest}
Let $f,\nabla g \in L^\infty $ and $D^{m-1}f\in L^2, D^m g\in L^2$ , $m$ being a positive integer. Then,
$$
||[D_\alpha(fg) - gD_\alpha f]|| \leq C (||f||_\infty ||D^m g|| + ||\nabla g||_\infty||D^{m-1} f||),
$$
where $\alpha$ is a multi-index with $|\alpha|\leq m$ and $C$ depends only on $m$.
\end{lemma}
We will also need the Fractional Leibniz rule, \cite[Proposition 2.1]{M84}.
\begin{lemma}
\label{fracleib}
Let $f,g\in L^\infty$ and $D^m f, D^m g \in L^2$, $m$ being a nonnegative integer. Then,
$$
||D_{\alpha} (fg)||\leq C\left( ||f||_\infty ||D^m g|| + ||g||_\infty||D^m f|| \right),
$$
where $\alpha$ is a multi-index with $|\alpha|\leq m$ and $C$ depends only on $m$.
\end{lemma}
\textit{Proof of Theorem \ref{hsenergyestimate}:}
Hereinafter, let $C$ be a constant whose exact value changes along the lines but it depends only on $s$ and $\beta$.
Let $\alpha$ be a multi-index with $1\leq |\alpha|\leq s+1$. Operating \eqref{hypMain}(i) with $D_\alpha$, we obtain,
\begin{align*}
D_\alpha\rho_t & = -D_\alpha (\nabla\cdot(\rho Q\ast \mathbf{u}))\\
& = -[D_\alpha (\nabla\cdot(\rho Q\ast\mathbf{u})) - Q\ast\mathbf{u}\cdot \nabla(D_\alpha\rho)]- Q\ast\mathbf{u}\cdot \nabla(D_\alpha\rho).
\end{align*}
Consequently, we obtain
\begin{align*}
\frac{1}{2}\frac{d}{dt}\int (D_\alpha \rho )^2 \, dx & = -\int D_\alpha\rho [D_\alpha (\nabla\cdot(\rho Q\ast\mathbf{u})) - Q\ast\mathbf{u}\cdot \nabla(D_\alpha\rho)]\, dx \\
& - \int D_\alpha\rho\, [ Q\ast\mathbf{u}\cdot \nabla D_\alpha \rho] \, dx\\
& =: \RN{1} + \RN{2}.
\end{align*}   
We estimate each term individually. Integration by parts on $\RN{2}$, we have
\begin{align*}
\RN{2} & = -\frac{1}{2}\int\nabla(D_\alpha\rho)^2\cdot(Q\ast\mathbf{u})\, dx \\
& = \frac{1}{2} \int (D_\alpha\rho)^2\, \nabla\cdot (Q\ast\mathbf{u})\, dx.
\end{align*}
Noting that $Q\in W^{1,1}$, we have $||\nabla\cdot Q\ast\mathbf{u}||\leq C||\mathbf{u}||_\infty$. Therefore,
$$
|\RN{2}|\leq C ||\mathbf{u}||_\infty ||D_\alpha\rho||^2\leq C||\mathbf{u}||_\infty ||\nabla\rho ||_{H^s}^2.
$$
Next, we proceed to estimate $\RN{1}$ by Cauchy-Schwarz and a subsequent application of Lemma \ref{commest} with $f=\rho$, $g = Q\ast \mathbf{u}$ (component-wise) and $m=s+2$:
\begin{align*}
|\RN{1}| & \leq   ||D_\alpha \rho|| \left|\left| D_\alpha \nabla\cdot (\rho Q\ast \mathbf{u}) - (Q\ast \mathbf{u})\cdot\nabla D_\alpha \rho   \right|\right| \\
& \leq C ||D_\alpha\rho|| \left( ||D^{s+1}\mathbf{u}||\, ||\rho||_\infty + ||D^{s+1} \rho||\, ||\mathbf{u}||_\infty\right)\\
& \leq C ||\nabla\rho||_{H^s} \left( ||\nabla\mathbf{u}||_{H^s}||\rho||_\infty + ||\nabla\rho||_{H^s}||\mathbf{u}||_\infty \right)\\
& \leq C (||\mathbf{u}||_\infty + ||\rho||_\infty ) \left( ||\nabla \rho||_{H^s}^2 + ||\nabla \mathbf{u}||_{H^s}^2 \right).
\end{align*}
Putting together bounds for both $\RN{1}$ and $\RN{2}$ and further adding all inequalities for $\alpha$ such that $1\leq\alpha\leq s+1$, we obtain
\begin{align}
\label{rhohsbound}
& \frac{d}{dt}|| \nabla\rho(t)||^2_{H^s} \leq C \left( ||\mathbf{u}(t)||_\infty + ||\rho(t)||_\infty \right) \left( ||\nabla\rho(t)||^2_{H^s}+ ||\nabla\mathbf{u}(t)||^2_{H^s} \right).
\end{align}
In the above step and throughout the rest of this paper, we show the dependence on time after taking the norm in the space variable as above since it is clear from the context. Similar line of arguments can be followed to find the analogous bound on the derivative of $||\nabla\mathbf{u}||_{H^s}^2$, using \eqref{hypMain}(ii). Let us denote a fixed $i^{th}$ component of $\mathbf{u}$ as $u$. We have,
\begin{align*}
D_\alpha u_t & = D_\alpha[\rho(Q\ast u - u) ] - D_\alpha (\mathbf{u}\cdot \nabla u)\\
& = D_\alpha[\rho(Q\ast u - u) ] - [ D_\alpha(\mathbf{u}\cdot\nabla u) - \mathbf{u}\cdot \nabla D_\alpha u] - \mathbf{u}\cdot \nabla D_\alpha u.
\end{align*}
Consequently,
\begin{align*}
\frac{1}{2}\frac{d}{dt}\int (D_\alpha u)^2 \, dx & =  \int D_\alpha u\,  D_\alpha [\rho (Q\ast u - u)] \, dx \\
& - \int D_\alpha u[D_\alpha(\mathbf{u}\cdot\nabla u)-\mathbf{u}\cdot\nabla D_\alpha u]\, dx - \int (D_\alpha u)\mathbf{u}\cdot\nabla D_\alpha u\, dx \\
& =: \RN{1} + \RN{2} + \RN{3}. 
\end{align*}
We start with bounding $\RN{3}$, 
\begin{align*}
\RN{3} & = -\frac{1}{2}\int \nabla (D_\alpha u)^2 \cdot \mathbf{u}\, dx\\
& = \frac{1}{2}\int (D_\alpha u)^2\, \nabla\cdot\mathbf{u}\, dx.
\end{align*}
Therefore, $|\RN{3}|\leq C ||D_\alpha u||^2||\nabla\mathbf{u}||_\infty \leq C ||\nabla u||_{H^s}^2||\nabla\mathbf{u}||_\infty $.
Next we bound $\RN{1}$,
\begin{align*}
|\RN{1}| & \leq ||D_\alpha u||\, \left|\left| D_\alpha[\rho(Q\ast u-u)] \right|\right| \\
& \leq C ||D_\alpha u|| \left( ||\rho||_\infty ||D^{s+1} u|| + ||D^{s+1} \rho ||\, ||\mathbf{u}||_\infty  \right)\\
& \leq C ||\nabla u||_{H^s} \left( ||\rho||_\infty ||\nabla u||_{H^s} + ||\nabla \rho ||_{H^s}\, ||\mathbf{u}||_\infty  \right)\\
& \leq C \left( ||\rho||_\infty + ||\mathbf{u}||_\infty \right)\left( ||\nabla\rho||_{H^s}^2 + ||\nabla u||_{H^s}^2 \right).
\end{align*}
where we used Lemma \ref{fracleib} with $f=\rho$, $g=Q\ast u-u$ and $m=s+1$. To bound $\RN{2}$, we make use of Lemma \ref{commest} with $f = \nabla u$ and $g = \mathbf{u}$. Consequently,
\begin{align*}
|\RN{2}| & \leq C ||D_\alpha u||\left|\left| D_\alpha(\mathbf{u}\cdot\nabla u) - \mathbf{u}\cdot\nabla D_\alpha u \right|\right|\\
& \leq C ||D_\alpha u||\left( ||\nabla u||_\infty ||D^{s+1} \mathbf{u}|| + ||\nabla\mathbf{u}||_\infty ||D^{s+1} u|| \right)\\
& \leq C ||D_\alpha u||\, ||\nabla \mathbf{u}||_\infty ||\nabla \mathbf{u}||_{H^s}\\
& \leq C ||\nabla\mathbf{u}||_\infty ||\nabla\mathbf{u}||_{H^s}^2.
\end{align*}
Combining all the three bounds for $\RN{1}, \RN{2}, \RN{3}$, and adding these inequalities for all the components of $\mathbf{u}$ and $\alpha$ such that $1\leq |\alpha|\leq s+1$, we immediately obtain,
\begin{align}
\label{uhsbound}
& \frac{d}{dt}||\nabla \mathbf{u}(t)||^2_{H^s} \leq C\left( ||\nabla\mathbf{u}(t)||_\infty + ||\rho(t)||_\infty + ||\mathbf{u}(t)||_\infty \right) \left( ||\nabla\rho(t)||^2_{H^s}+ ||\nabla\mathbf{u}(t)||^2_{H^s} \right).
\end{align}
Finally adding (\ref{rhohsbound}) and (\ref{uhsbound}), 
\begin{align*}
\frac{dY}{dt} \leq C\left( ||\rho(t)||_\infty + ||\nabla\mathbf{u}(t)||_\infty + ||\mathbf{u}(t)||_\infty \right)Y,
\end{align*}
where $C$ depends only on $\beta$ and $s$. Upon integration, we get the desired result.\\
\hfill \qed

\section{Local well-posedness} 
Here, we will prove Theorem \ref{local}. The idea is to use the typical iteration technique with the given initial data as the first guess. We will need the following lemma.
\begin{lemma}
\label{intp}
Suppose $\{f^k(t,x)\}_{k=1}^\infty$ be a sequence of continuous functions with  
$$\sup_{[0,T]} ||\nabla f^k(t,\cdot )||_{H^s(\mathbb{R}^N)}\leq \kappa
$$ 
with $\kappa,T$ being positive constants and $s\geq 1$ is an integer. Also 
$$
f^k\to f \text{ in } C([0,T];L^\infty (\mathbb{R}^N)).
$$ 
Then $D^1 f\in C([0,T];H^{s-1}(\mathbb{R}^N))$ and
$$
D^1 f^k\to D^1 f\ in\ C([0,T];H^{s-1}(\mathbb{R}^N)). 
$$
\end{lemma}
\begin{proof}
Let $U=B_M(0)$ (ball of radius $M$ around origin) for $M>0$ fixed. Throughout this proof, $C_1$ is a constant that may change along the lines but only depend on $s$ and $M$. We will use the following interpolation inequality \cite[Theorem 7.28]{GT01}: for all $\epsilon >0$ and any multi-index $\alpha$ with $1\leq |\alpha| \leq s$,
\begin{align*}
||D_\alpha(f^m-f^n)||_{L^2(U)} & \leq \epsilon ||f^m-f^n||_{H^{s+1}(U)} + \frac{C_1}{\epsilon^{\frac{|\alpha|}{s+1-|\alpha|}}}||f^m-f^n||_{L^2(U)}\\
& \leq \epsilon||\nabla (f^m-f^n)||_{H^s(U)} + \epsilon ||f^m - f^n||_{L^2(U)} + \frac{C_1}{\epsilon^{\frac{|\alpha|}{s+1-|\alpha|}}}||f^m-f^n||_{L^2(U)}\\
& \leq 2\epsilon\kappa + \left( \epsilon + \frac{C_1}{\epsilon^{\frac{|\alpha|}{s+1-|\alpha|}}} \right)||f^m-f^n||_{L^2(U)}.
\end{align*}
Therefore, for any $\epsilon < 1$,
\begin{align*}
\sup_{[0,T]} ||D_\alpha(f^m-f^n)||_{L^2(U)} & \leq 2\epsilon \kappa + C_1\left( \epsilon + \frac{1}{\epsilon^{s+1}} \right) \sup_{[0,T]} ||(f^m-f^n)(t)||_\infty \\
& \leq 2\epsilon \kappa + \frac{C_1}{\epsilon^{s+1}} \sup_{[0,T]} ||(f^m-f^n)(t)||_\infty. 
\end{align*}
Letting $m,n\to\infty$, we get
$$
\lim_{m,n\to\infty} \sup_{[0,T]}||D_\alpha(f^m-f^n)(t)||_{L^2(U)}\leq 2\epsilon\kappa ,\quad \forall \epsilon\in (0,1).
$$
Since $1\leq |\alpha|\leq s$, we conclude that $\{D^1 f^k\}_{k=1}^\infty$ is Cauchy in $C([0,T];H^{s-1}(U))$. Consequently, $D^1 f^k\to D^1 f$ in $C([0,T];H^{s-1}(U))$. Lastly, we complete the proof by showing that the sequence is indeed Cauchy in $C([0,T];H^{s-1}(\mathbb{R}^N))$. To this end, fix a $\delta >0$. For any multi-index $\alpha$ with $1\leq |\alpha| \leq s$,
\begin{align*}
\int_{\mathbb{R}^N}|D_\alpha(f^m-f^n)|^2 dx & = \int_{B_M(0)} \!\!\!\!\!\! |D_\alpha(f^m-f^n)|^2 dx + \int_{\mathbb{R}^N\backslash B_M(0)}\!\!\!\!\!\! \!\!\!\!\!\! |D_\alpha(f^m-f^n)|^2 dx \\
& \leq \int_{B_M(0)} \!\!\!\!\!\! |D_\alpha(f^m-f^n)|^2 dx + 2\left(\int_{\mathbb{R}^N\backslash B_M(0)}\!\!\!\!\!\! \!\!\!\!\!\! |D_\alpha f^m|^2 +|D_\alpha f^n|^2 dx\right) \\
& \leq \sup_{[0,T]}\int_{B_{M_0}(0)} \!\!\!\! \!  |D_\alpha (f^m-f^n)(t)|^2 dx + \delta, 
\end{align*}
for sufficiently large $M_0$. Letting $m,n\to\infty$,
$$
\lim_{m,n\to\infty}\sup_{[0,T]}\int_{B_{M_0}(0)} \!\!\!\! \! \!\!\! \! |D_\alpha (f^m-f^n)(t)|^2\, dx\leq \delta,
$$
and this holds for any $\delta >0$. Hence, $D^1 f^k\to D^1 f$ in $C([0,T];H^{s-1}(\mathbb{R}^N))$.
\end{proof}
The rest of the proof is divided into four steps. We begin by setting up the iteration system and introducing some notation.
Let $\rho^0 = \rho_0 (x)$ and $\mathbf{u^0} = \mathbf{u_0}(x)$. For $k\geq 0$, we update $(\rho^k, \mathbf{u}^k)$ 
by the following scheme,  
\begin{align}
\label{itersch}
\begin{aligned}
& \rho_t^{k+1} + \nabla\cdot\left( \rho^{k+1} Q\ast \mathbf{u}^k \right) = 0,\\
& \mathbf{u}_t^{k+1} + (\mathbf{u}^k\cdot\nabla ) \mathbf{u}^{k+1} = \rho^{k}\left(  Q\ast \mathbf{u}^{k+1} - \mathbf{u}^{k+1} \right),
\end{aligned}
\end{align}
with $\rho^{k+1}(0,x) = \rho_0 (x)$ and $\mathbf{u}^{k+1}(0,x) = \mathbf{u_0}(x)$.  The $N+1$ equations in this system 
are decoupled and, therefore, for given smooth functions $\rho^k, \mathbf{u}^k$ and smooth initial data, there always exists a unique classical solution for this system. It suffices to estimate the solution in terms of $\rho^k, \mathbf{u}^k$. In all the four steps of the proof, we denote 
\begin{align*}
& \rho:=\rho^{k+1},\quad \mathbf{u}:=\mathbf{u}^{k+1},\quad w:=\rho^k,\quad \mathbf{v}:=\mathbf{u}^k,\\ 
& \tilde{\rho} := \rho^{k+1}-\rho^k,\quad \tilde{\mathbf{u}}:= \mathbf{u} ^{k+1}-\mathbf{u} ^k,\quad \tilde{w} := \rho^{k}-\rho^{k-1},\quad \tilde{\mathbf{v}} := \mathbf{u} ^{k}-\mathbf{u} ^{k-1},
\end{align*}
for the sake of notational simplicity. Also, $u,v,\tilde{u},\tilde{v}$ will be used to denote a single component of the corresponding vector in boldface.

\textbf{Step 1:} \textit{Uniform Bounds in High Norm.}  
For fixed $R>0$, assume that the  initial data satisfy 
$$
\max\{ ||\rho_0||_\infty,  ||\nabla\rho_{0} ||_{H^s}, ||\mathbf{u_0}||_\infty,  ||\nabla \mathbf{u_{0}} ||_{H^s} \} \leq R/2.
$$
By induction, we will show that there exists $T>0$ such that for all $k\geq 1$, 
\begin{align}
\label{munibd}
\max\left\{ \sup_{[0,T]} ||\rho^k(t)||_\infty , \sup_{[0,T]} ||\nabla \rho^k(t) ||_{H^s}, \sup_{[0,T]} ||\mathbf{u}^k(t)||_\infty , \sup_{[0,T]} ||\nabla\mathbf{u}^k(t) ||_{H^s} \right\}\leq R.
\end{align}
 Assume this already holds for $k$, i.e.,   
\begin{align}
\label{boundassumpn}
\max\left\{ \sup_{[0,T]} ||w(t)||_\infty , \sup_{[0,T]} ||\nabla w(t) ||_{H^s}, \sup_{[0,T]} ||\mathbf{v}(t)||_\infty ,\sup_{[0,T]} ||\nabla \mathbf{v}(t) ||_{H^s} \right\}\leq R.
\end{align}
Since $\mathbf{v}=\mathbf{u}^k$ is Lipschitz continuous in $x$, we can apply Lemma \ref{transbd} to
\eqref{itersch}(i) to obtain 
\begin{align}\label{winfb}
\rho \leq ||\rho_0||_\infty e^{\beta R t} \leq R,
\end{align}
where we used \eqref{boundassumpn}, and took $t\in [0,T_1]$ with $T_1 =  \frac{1}{\beta R}\ln\left( \frac{R}{||\rho_0||_\infty}  \right)$.

Hereinafter, $C$ is a constant which changes along lines but depends on $s$, $\beta$ and $R$ only. For $H^s$ bound, we study the evolution of the $D_\alpha \rho$ for a multi-index $\alpha$ with $1\leq |\alpha|\leq s+1$, similar to what we did in the proof of Theorem \ref{hsenergyestimate} above. Using \eqref{itersch}(i),
\begin{align*}
\frac{1}{2}\frac{d}{dt}\int |D_\alpha\rho|^2\, dx & = - \int (D_\alpha\rho)  D_\alpha (\nabla\cdot(\rho Q\ast \mathbf{v}))\, dx\\
& = -\int D_\alpha\rho [D_\alpha (\nabla\cdot(\rho Q\ast\mathbf{v})) - Q\ast\mathbf{v}\cdot \nabla(D_\alpha\rho)]\, dx \\
& - \int D_\alpha\rho\, [ Q\ast\mathbf{v}\cdot \nabla D_\alpha \rho] \, dx\\
& \leq C ||D_\alpha\rho||\left(||\rho||_\infty||D^{s+1} \mathbf{v}|| + || \mathbf{v}||_\infty ||D^{s+1} \rho||\right) + \frac{\beta}{2}||\mathbf{v}||_\infty ||D_\alpha\rho||^2\\
& \leq C\left( ||\rho||_\infty + ||\mathbf{v}||_\infty \right) \left( ||\nabla\rho||^2_{H^s} + ||\nabla\mathbf{v}||^2_{H^s} \right)
\end{align*}
We used Lemma \ref{commest} to obtain the second to last inequality and used \eqref{boundassumpn} and Young's inequality for the last one. Consequently, 
$$
\frac{d}{dt}||D_\alpha\rho(t)||^2 \leq C\left(||\nabla\rho(t)||^2_{H^s} + 1\right).
$$
Adding the inequalities for all multi-indices $\alpha$, we obtain 
$$
\frac{d}{dt} ||\nabla\rho(t)||_{H^s}^2 \leq C(||\nabla\rho(t)||_{H^s}^2 + 1), 
$$
which upon integration gives 
\begin{align*}
||\nabla\rho(t)||_{H^s}^2 \leq e^{C t}(||\nabla\rho_{0}||_{H^s}^2 + 1)-1.
\end{align*}
Choosing $T_2 = \frac{1}{C}\ln\left( \frac{1+R^2}{1+||\nabla\rho_{0}||_{H^s}^2}  \right)$, we ensure,
\begin{align}
\label{wunibd}
\sup_{[0,T_2]}||\nabla\rho(t)||_{H^s}\leq R.
\end{align}
As for $\mathbf{u}$ we have from a single equation in \eqref{itersch}(ii), 
$$
u_t + \mathbf{v}\cdot(\nabla u) = w(Q\ast u-u).
$$
A very similar argument as in proof for the maximum principle (Proposition \ref{bdds}) leads to the following conclusion,
\begin{align}
\label{vinfb}
||\mathbf{u}(t)||_\infty \leq ||\mathbf{u_0}||_\infty\leq R,\quad \forall t\geq 0.
\end{align}
Next, for $1\leq |\alpha|\leq s+1$ consider,
\begin{align*}
\frac{1}{2}\frac{d}{dt}\int |D_\alpha u|^2\, dx & = \int  (D_\alpha u) D_\alpha (-\mathbf{v}\cdot\nabla u + w(Q\ast u-u))\, dx\\
& = -\int (D_\alpha u) [D_\alpha(\mathbf{v}\cdot\nabla u)-\mathbf{v}\cdot \nabla D_\alpha u]\,dx + \frac{1}{2}\int (\nabla\cdot \mathbf{v}) (D_\alpha u)^2\, dx\\
& + \int D_\alpha u\, D_\alpha (w(Q\ast u-u))\, dx\\
& =: \RN{1}+\RN{2}+\RN{3}. 
\end{align*}
We use Lemma \ref{commest} to control $\RN{1}$, Sobolev embedding theorem to control $||\nabla\cdot v||_\infty$ in $\RN{2}$ and Lemma \ref{fracleib} to control $\RN{3}$. Consequently,
\begin{align*}
\frac{d}{dt} ||(D_\alpha u)(t)||^2 \leq C(||\nabla u(t)||_{H^s}^2 + 1).
\end{align*} 
Summing over all multi-indices $\alpha$ and components of $\mathbf{u}$, we obtain
\begin{align*}
\frac{d}{dt} ||\nabla\mathbf{u}(t)||_{H^s}^2 \leq C(||\nabla \mathbf{u}(t)||_{H^s}^2 + 1).
\end{align*}
This yields 
\begin{align*}
||\nabla\mathbf{u}(t)||_{H^s}^2 \leq e^{C t}(||\nabla\mathbf{u_{0}}||_{H^s}^2 + 1)-1.
\end{align*}
Choosing $T_3 = \frac{1}{C}\ln\left( \frac{1+R^2}{1+||\nabla \mathbf{u_{0}}||_{H^s}^2}  \right)$, we ensure, 
\begin{align}
\label{vunibd}
\sup_{[0,T_3]}||\nabla\mathbf{u}(t)||_{H^s}\leq R.
\end{align}
Hence, for $T=\min \{T_1,T_2,T_3\}$, and using \eqref{winfb}-\eqref{vunibd}, we have 
\begin{align*}
\max\left\{ \sup_{[0,T]} ||\rho(t)||_\infty , \sup_{[0,T]} ||\nabla\rho(t) ||_{H^s}, \sup_{[0,T]} ||\mathbf{u}(t)||_\infty , \sup_{[0,T]} ||\nabla \mathbf{u}(t) ||_{H^s} \right\}\leq R.
\end{align*}
By induction, we obtain \eqref{munibd}, therefore, concluding Step 1.\\

\textbf{Step 2:} \textit{Cauchy Sequence in Infinity Norm.} Here we prove $(\rho^k, \mathbf{u}^k)$ is a Cauchy sequence 
in order to guarantee that there exists $(\rho, \mathbf{u})$ as its limit as $k \to \infty$. More precisely we have the following lemma. 
\begin{lemma}\label{lemc}
There exists $T^*\leq T$ such that 
\begin{align}
\label{contrainfn}
\begin{aligned}
& \max\left\{  \sup_{[0,T^*]} ||\rho^{k+1}(t) - \rho^{k}(t)||_\infty,\,   \sup_{[0,T^*]} ||\mathbf{u}^{k+1}(t) - \mathbf{u}^{k}(t)||_\infty  \right\}\\
\leq & \ \frac{1}{2} \max\left\{  \sup_{[0,T^*]} ||\rho^{k}(t) - \rho^{k-1}(t)||_\infty,\,  \sup_{[0,T^*]} ||\mathbf{u}^{k}(t) - \mathbf{u}^{k-1}(t)||_\infty  \right\},
\end{aligned}
\end{align}
for $k=1, 2, \cdots.$
\end{lemma}
Then in an augmented Banach space 
$$
S = \left\{ (w,\mathbf{v}): (w,\mathbf{v})\in (C([0,T]; C_b(\mathbb{R}^N)))^{N+1},\ w\geq 0  \right\}
$$ 
with
$$
||(w,\mathbf{v})||_S = \max\{ \sup_{[0,T]} ||w(t)||_\infty,\, \sup_{[0,T]} ||\mathbf{v}(t)||_\infty \},
$$ 
$(\rho^k, \mathbf{u}^k)$ is convergent in the sense that $(\rho^k ,\mathbf{u}^k)\to (\rho ,\mathbf{u})$ in $C([0,T^*];L^\infty (\mathbb{R}^N))$.
Hence, the only remaining concern after this step would be to show that the limit functions $(\rho, \mathbf{u})$ is indeed a classical solution to \eqref{hypMainwic}, with higher regularity of this solution. We perform the aforementioned analysis in the next step. Here, we will show \eqref{contrainfn}. 
\begin{proof}
Taking difference \eqref{itersch}(i) for $k+1$ and $k$, we obtain 
$$
\tilde{\rho}_t + \nabla\cdot(\tilde{\rho}Q\ast \mathbf{v}) = -\nabla\cdot (wQ\ast \tilde{\mathbf{v}}).
$$
Using Lemma \ref{transbd}, we have  
$$
||\tilde{\rho}(t)||_\infty \leq \int_0^t e^{\int_s^t ||\nabla\cdot Q\ast\mathbf{v}(\tau)||_\infty d\tau}||(\nabla w)\cdot Q\ast \tilde{\mathbf{v}} + w\nabla\cdot Q\ast \tilde{\mathbf{v}}||_\infty ds,
$$
where we used the fact that $\tilde{\rho}(0) = 0$. Consequently, using \eqref{munibd},
\begin{align*}
||\tilde{\rho}(t)||_\infty & \leq \int_0^t e^{\beta R(t-s)}(||(\nabla w \cdot  Q\ast\tilde{\mathbf{v}})(s) ||_\infty + ||(w\nabla \cdot Q\ast\tilde{\mathbf{v}})(s)||_\infty)\, ds\\
& \leq C \sup_{[0,T]} ||\tilde{\mathbf{v}}(t)||_\infty (e^{\beta Rt}-1),\quad t\in [0,T].
\end{align*}
Here we used the Sobolev embedding theorem to control $\nabla w$ along with \eqref{munibd}. Let $T_1^* = \min\{T,  \frac{\ln(1/2C + 1)}{\beta R}\}$, so that 
\begin{align}
\label{rhotvt}
\sup_{[0,T_1^*]} ||\tilde{\rho}(t)||_\infty \leq \frac{1}{2} \sup_{[0,T_1^*]} ||\tilde{\mathbf{v}}(t)||_\infty.
\end{align}
For $\tilde{u}$, we take the difference of \eqref{itersch}(ii) between $k+1$ and $k$ to obtain,
\begin{align*}
& \tilde{\mathbf{u}}_t + (\mathbf{v}\cdot\nabla) \tilde{\mathbf{u}} = -(\tilde{\mathbf{v}}\cdot\nabla ) \mathbf{v} + \tilde{w}(Q\ast \mathbf{v}- \mathbf{v}) + w(Q\ast\tilde{\mathbf{u}}-\tilde{\mathbf{u}}).
\end{align*}
Along the characteristic path $\{(t,x): \frac{dx}{dt} = \mathbf{v},\, x(0) = \alpha,\, \alpha\in\mathbb{R}^N \}$,
$$
\frac{d\tilde{\mathbf{u}}}{dt} = -(\tilde{\mathbf{v}}\cdot\nabla ) \mathbf{v} + \tilde{w}(Q\ast \mathbf{v}- \mathbf{v}) + w(Q\ast\tilde{\mathbf{u}}-\tilde{\mathbf{u}}),\quad \tilde{\mathbf{u}}(0) = 0.
$$
From this we obtain the following,
\begin{align}
\label{doubin}
-C(||\tilde{\mathbf{v}}(t)||_\infty + ||\tilde{w}(t)||_\infty)\mathbf{1} + w(Q\ast\tilde{\mathbf{u}}-\tilde{\mathbf{u}})\leq \frac{d\tilde{\mathbf{u}}}{dt}\leq C(||\tilde{\mathbf{v}}(t)||_\infty + ||\tilde{w}(t)||_\infty)\mathbf{1} + w(Q\ast\tilde{\mathbf{u}}-\tilde{\mathbf{u}}), 
\end{align}
where $\mathbf{1}$ is an $N\times 1$ column vector of ones.  In order to bound $\tilde{\mathbf{u}}$ we introduce an auxiliary problem  
\begin{align}
\label{etaaux}
\frac{d\eta}{dt} = C(||\tilde{\mathbf{v}}(t)||_\infty + ||\tilde{w}(t)||_\infty),\quad \eta(0) = 0.
\end{align}
This will allow us to prove 
\begin{align}
\label{ue}
-\eta(t)\mathbf{1} \leq \tilde{\mathbf{u}}(t) \leq \eta(t)\mathbf{1}, \quad t\in [0, T]. 
\end{align}
Based on this we see that 
\begin{align*}
||\tilde{\mathbf{u}}(t)||_\infty & \leq Ct\left(\sup_{[0,t]}||\tilde{\mathbf{v}}(s)||_\infty + \sup_{[0,t]}||\tilde{w}(s)||_\infty\right)\\
& \leq 2Ct  \max\left\{\sup_{[0,t]}||\tilde{\mathbf{v}}(s)||_\infty,\, \sup_{[0,t]}||\tilde{w}(s)||_\infty\right\} .
\end{align*}
Let $T_2^* = \min\{T, (4C)^{-1}\}$, we finally have,
\begin{align}
\label{utvtwt}
\sup_{[0,T_2^*]} ||\tilde{\mathbf{u}}(t)||_\infty\leq \frac{1}{2} \max\left\{\sup_{[0,T_2^*]}||\tilde{\mathbf{v}}(t)||_\infty,\, \sup_{[0,T_2^*]}||\tilde{w}(t)||_\infty\right\}. 
\end{align}
Combining this with \eqref{rhotvt}, we finally obtain \eqref{contrainfn} for $T^*=\min\{ T_1^*,T_2^*\}$.

Finally we return to prove (\ref{ue}). Taking difference of second inequality in \eqref{doubin} and \eqref{etaaux}$\mathbf{1}$, we obtain
$$
\frac{d(\tilde{\mathbf{u}}-\eta\mathbf{1})}{dt} \leq w(Q\ast\tilde{\mathbf{u}} - \tilde{\mathbf{u}}). 
$$
Using a substitution $\boldsymbol{\xi} = \tilde{\mathbf{u}} - \eta\mathbf{1}$ and the fact that $\eta$ is independent of the space variable, we have the following simple inequality,
\begin{align*}
\frac{d\boldsymbol\xi}{dt}\leq w(Q\ast\boldsymbol\xi -\boldsymbol\xi),
\end{align*}
with $\boldsymbol\xi (0) = \mathbf{0}$.
Using the same argument as in Proposition \ref{bdds}, we obtain that
\begin{align*}
\boldsymbol\xi(t,\cdot )\leq \sup\boldsymbol\xi (0)= 0.
\end{align*}
Hence,
\begin{align*}
\tilde{\mathbf{u}}(t)\leq\eta (t)\mathbf{1}, \ \forall t\geq 0.
\end{align*}
Similarly, by taking the sum of the first inequality in \eqref{doubin} and \eqref{etaaux}, and proceeding along the same line of arguments, we have
\begin{align*}
\tilde{\mathbf{u}}(t)\geq -\eta (t)\mathbf{1}, \ \forall t\geq 0.
\end{align*}
\end{proof}

\textbf{Step 3:} \textit{Higher Regularity of $(\rho ,\mathbf{u})$.} In this step, we show $D^1 \rho \in C([0,T];H^s(\mathbb{R}^N))$. Similar arguments will follow for $D^1 \mathbf{u}$.
Using Lemma \ref{intp} on Steps 1 and 2 with $f^k = \rho^k$, we have
\begin{align}
\label{ass1}
D^1 \rho^k\to D^1 \rho\text{ in } C([0,T];H^{s-1}(\mathbb{R}^N)). 
\end{align}

Next, let $\phi\in H^{-s}(\mathbb{R}^N)$ and $\{\phi_l\}_{l=1}^\infty \subseteq H^{-(s-1)}(\mathbb{R}^N)$ with $\phi^l\to\phi$ in $H^{-s}(\mathbb{R}^N)$. This is possible because $H^{-(s-1)}(\mathbb{R}^N)$ is dense in $H^{-s}(\mathbb{R}^N)$. Denoting $\langle\cdot ,\cdot\rangle$ as the pairing through $L^2(\mathbb{R}^N)$ inner product and for any $\epsilon >0$, we have, 
\begin{align*}
\left|\langle D^1 \rho^m (t)-D^1 \rho^n (t),\, \phi\rangle\right|& \leq \left|\langle D^1\rho^m(t)-D^1\rho^n(t),\, \phi-\phi^l\rangle\right| + \left|\langle D^1\rho^m(t)-D^1\rho^n(t),\, \phi^l\rangle\right|\\
& \leq ||D^1\rho^m(t)-D^1\rho^n(t)||_{H^s}||\phi - \phi_l||_{H^{-s}} + \left|\langle D^1\rho^m(t)-D^1\rho^n(t),\, \phi^l\rangle\right|\\
& \leq C ||\phi - \phi_l||_{H^{-s}}+ \left|\langle D^1\rho^m(t)-D^1\rho^n(t),\, \phi^l\rangle\right|\\
& \leq \epsilon + \left|\langle D^1\rho^m(t)-D^1\rho^n(t),\, \phi^l\rangle\right|,
\end{align*}
for sufficiently large $l$. Using \eqref{ass1} and letting $m,n\to\infty$, we obtain $\langle D^1\rho^k(t),\phi\rangle\to \langle D^1\rho(t),\phi\rangle$ uniformly in time. And since uniform limit of continuous functions is continuous, $D^1\rho\in C_w([0,T];H^s(\mathbb{R}))$, i.e., $D^1\rho(t)$ is continuous in weak $H^s$ topology.

Using this, we prove right continuity for the function $||\nabla\rho(t)||_{H^s}$ in $[0,T)$. Without loss of generality, we will show right continuity at $t=0$. From weak continuity, we have
$$
||\nabla\rho_{0}||_{H^s}\leq \liminf_{t\to 0^+} ||\nabla\rho(t)||_{H^s}.
$$ 
Also, Theorem \ref{hsenergyestimate} implies that $\nabla\rho,\nabla\mathbf{u}\in L^\infty([0,T];H^s(\mathbb{R}^N))$. Using this and a minor change in proof of Theorem \ref{hsenergyestimate}, we can have an energy estimate only for $\rho$ which gives,
$$
\limsup_{t\to 0^+} ||\nabla\rho(t)||_{H^s}\leq ||\nabla\rho_{0}||_{H^s}\, .
$$ 
Combining the two inequalities, we have right continuity. For continuity from the left in $(0,T]$, we consider the time reversed problem to \eqref{hypMain} by making the substitution $t\to T-t$. All the relevant arguments hold for the time reversed solution $(\rho(T-t,x),u(T-t,x))$. Therefore, $\nabla\rho ,\nabla\mathbf{u} \in C([0,T];H^s(\mathbb{R}^N))$.

\textbf{Step 4:} \textit{Uniqueness.} Let $(\Delta \rho,\Delta \mathbf{u}): = (\rho_1 - \rho_2, \mathbf{u_1} - \mathbf{u_2})$ with $\rho_1 ,\rho_2 ,\mathbf{u_1}, \mathbf{u_2}$ being solutions to \eqref{hypMainwic}. By following the same line of argument as in Step 2, it can be shown that for a sufficiently small $T>0$,
$$
\max\left\{  \sup_{[0,T]} ||\Delta\rho (t)||_\infty,\,   \sup_{[0,T]} ||\Delta \mathbf{u}(t)||_\infty  \right\}\leq  \ \frac{1}{2} \max\left\{  \sup_{[0,T]} ||\Delta\rho(t)||_\infty,\,  \sup_{[0,T]} ||\Delta \mathbf{u}(t)||_\infty  \right\}.
$$  
Therefore, $\rho_1 = \rho_2$ and $\mathbf{u_1} = \mathbf{u_2}$. This concludes the proof of Theorem \ref{local}.

\section*{Acknowledgement}
This research was supported in part by NSF grant DMS1812666. 

\bibliographystyle{abbrv}

\end{document}